\documentclass{article}
\usepackage[hidelinks]{hyperref}
\usepackage{geometry}
\usepackage{amsmath,amsthm}
\usepackage{amssymb}
\usepackage{graphicx, epstopdf}
\usepackage{authblk}

\newtheorem{theorem}{Theorem}[section]
\newtheorem{lemma}{Lemma}[section]
\newtheorem{proposition}{Proposition}[section]
\newtheorem{corollary}{Corollary}[section]

\newcommand{\real}{\mathbb{R}}

\title{On the Properties of Convex Functions over Open Sets}
\author{Yoel Drori}
\affil{Google LLC, Mountain View, CA}

\begin{document}
\maketitle

    \begin{abstract}
        We consider the class of smooth convex functions defined over an open convex set. We show that this class is essentially different than the class of smooth convex functions defined over the entire linear space by exhibiting a function that belongs to the former class but cannot be extended to the entire linear space while keeping its properties. We proceed by deriving new properties of the class under consideration, including an inequality that is strictly stronger than the classical Descent Lemma.
    \end{abstract}

\section{Introduction}

In this note we consider the class of differentiable functions that are convex and $L$-smooth over a convex set. We recall that:
\begin{enumerate}
    \item[(a)] A differentiable function $f$ is called \emph{convex} over a convex set $C$ if 
    \[
        0 \leq f(y) - f(x) - \langle f^\prime(x), y-x\rangle, \quad \forall x,y\in C.
    \]
    \item[(b)] Given $L\geq 0$, a differentiable function $f$ is called \emph{$L$-smooth} over a set $C$ if
    \[
         \| f'(x) - f'(y)\| \leq L \|x - y\|, \quad \forall x,y\in C.
    \]

\end{enumerate}

A well-known property of $L$-smooth convex functions is the Descent Lemma~\cite{bauschke2011convex,beck2017first}:
\begin{theorem}[Descent Lemma]\label{T:descentLemma}
    Let $X\subseteq \real^d$ be an open set, and let $f:X \rightarrow \real$ be a differentiable function that is $L$-smooth and convex over a convex set $C \subseteq X$.
	Then for any $x,y \in C$,
	\begin{equation}\label{eq:descent_lemma}
	0 \leq f(y) - f(x) - \langle f^\prime(x), y-x\rangle \leq \frac{L}{2} \|x-y\|^2.
	\end{equation}
\end{theorem}

Note that the Descent Lemma does not make any requirements on $C$ beyond convexity, allowing for some degenerate examples such as the non-convex quadratic function $f(x,y)=xy$, which is nevertheless convex over the set $\{(x,y):y=0\}$.

A standard assumption that is made in this context is that $C$ is an open set. This eliminates the edge cases alluded above, and gives rise to additional basic properties of the class of function, e.g., when the function is twice differentiable, then $\nabla^2 f(x) \succeq 0$ for all $x\in C$. 
For additional properties of convex functions over open sets we refer the reader to the standard texts~\cite{ortega1970iterative,Book:Polyak,rockafellar2015convex}.



Another special and important case is the unconstrained setting, $C=\real^d$. In this setting, a fundamental property that will form the cornerstone of the forthcoming analysis is known:
\begin{theorem}[{\cite[Theorem~2.1.5]{Book:Nesterov}}]\label{T:basicinequnconstrained}
	Let $f:\real^d \rightarrow\real$ be an $L$-smooth convex function (for some $L > 0$). Then for any $x,y\in \real^d$, 
	\begin{equation}\label{eq:basicinequnconstrained}
	\frac{1}{2L} \|f^\prime(x)-f^\prime(y)\|^2 \leq f(y)-f(x)-\langle f^\prime(x), y-x\rangle.
	\end{equation}
\end{theorem}

The bound in Theorem~\ref{T:basicinequnconstrained} is \emph{realizable} in the following sense: suppose
\[
    \frac{1}{2L} \|g_x-g_y\|^2 \leq f_y-f_x-\langle g_x, y-x\rangle
\]
holds for some $x, y\in\real^d$, $f_x, f_y\in \real$ and $g_x, g_y \in \real^d$, then there exists an $L$-smooth and convex function $F:\real^d\rightarrow\real$ such that $F(x)=f_x$, $F'(x)=g_x$, $F(y)=f_y$ and $F'(y)=g_y$. 
This has been recently shown in~\cite{taylor2017smooth} and also in~\cite{azagra2017extension} in a more general case over Hilbert spaces. For constructive results we refer the reader to~\cite{drori2017exact} and~\cite{daniilidis2018explicit} (for the general case).






Realizable bounds capture the basic aspects of the class: they are unique given the information they involve and cannot be improved, making them ideal for studying and developing optimization methods for the class under consideration.
In case of functions that are convex over an open or arbitrary convex set, however, realizable bounds are not known, and it is the goal of this note, inspired by open questions raised in~\cite{de2017worst}, to take some steps towards finding such bounds.

The main results of this note are as follows.
\begin{enumerate}
    \item We show that property~\eqref{eq:basicinequnconstrained} does \emph{not} hold in general for functions whose domain is an open convex set
    by constructing a smooth convex function defined over an open set that does not satisfy~\eqref{eq:basicinequnconstrained} over its domain.
    This shows that the assumption in Theorem~\ref{T:basicinequnconstrained} on the domain of the function is  indeed required, giving negative answers to the questions raised in~\cite{de2017worst}.
    \item We show that property~\eqref{eq:basicinequnconstrained} \emph{does} hold for convex functions defined over an open convex set, given that the two points $x, y$ are close enough.
    \item We derive a strictly stronger version of the Descent Lemma under the additional assumption that the function is convex over an open convex set.
    \item We present a system of inequalities that holds for any $L$-smooth and convex function defined over an open set and is realizable by an $L$-smooth and convex function defined over a segment.
    We show how this system can be used to form bounds that can be approximated using standard numerical methods, and compare the resulting bounds to existing analytical bounds.
\end{enumerate}

\section{A counter-example}
We begin by constructing a 1-smooth and convex function defined over a half-plane that cannot be extended to the entire plane while keeping its smoothness and convexity properties.
The construction is done by creating a ``convex-spline'' which consists of the function $F_i:\real^2 \rightarrow \real$, $i=1,\dots,4$ defined by
\begin{align*}
	F_1(x_0, x_1) :=& \tfrac{1}{2}(x_0^2 + x_1^2), \\
	F_2(x_0, x_1) :=& \tfrac{1}{2}(x_0^2 + x_1^2) - \tfrac{1}{20}(3 x_0 - x_1 - \tfrac{1}{12})^2 \\& \left[\equiv \tfrac{1}{2}((x_0-\tfrac{3}{4})^2+ (x_1+\tfrac{1}{4})^2) -\tfrac{1}{20}(3 x_0 - x_1 - \tfrac{31}{12})^2 +\tfrac{1}{48}\right], \\
	F_3(x_0, x_1) :=& \tfrac{1}{2}((x_0-\tfrac{3}{4})^2+ (x_1+\tfrac{1}{4})^2) + \tfrac{1}{48}, \\
	F_4(x_0, x_1) :=& \tfrac{1}{2}((x_0-\tfrac{3}{4})^2+ (x_1+\tfrac{1}{4})^2) - \tfrac{1}{10}(x_0-2x_1-\tfrac{49}{48})^2 +\tfrac{1}{48}.
\end{align*}
The function $F(x_0, x_1): \{ (x_0, x_1): x_1 > -23/240\} \rightarrow \real$ is then defined by
\[
    F(x_0, x_1):= 
    \begin{cases}
        F_1(x_0, x_1), & 3 x_0 - x_1 \leq \tfrac{1}{12}, \\
        F_2(x_0, x_1), & \tfrac{1}{12} \leq 3 x_0 - x_1 \leq \tfrac{31}{12}, \\
        F_3(x_0, x_1), & \tfrac{31}{12} \leq 3 x_0 - x_1\,\&\, x_0-2x_1 \leq \tfrac{49}{48}, \\
        F_4(x_0, x_1), & \tfrac{49}{48} \leq x_0-2x_1.
    \end{cases}
\]
See Figure~\ref{fig:counter} for a contour plot depicting $F$. 

\begin{figure}
    \centering
    \includegraphics[scale=1]{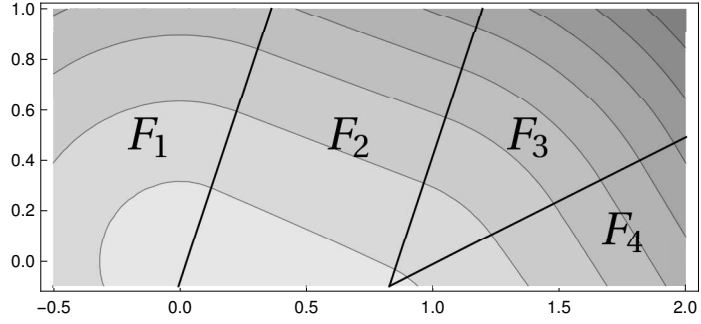}
    \caption{A contour plot of the function $F$.}
    \label{fig:counter}
\end{figure}

\begin{proposition}\label{P:counter}
\begin{enumerate}
    \item The function $F$ is convex and 1-smooth over its domain.
    \item The function $F$ does not satisfy~\eqref{eq:basicinequnconstrained} for $x=(0,0)^T$ and $y=(2,0)^T$.
\end{enumerate}
\end{proposition}

\begin{proof}
\begin{enumerate}
    \item It is straightforward to verify that the quadratic functions $F_1,\dots,F_4$ are convex and 1-smooth. Now, since $F$ is piecewise convex and continuously differentiable it follows that $F$ is convex (see~\cite[Corollary~5.5]{bauschke2016convexity}).
    Finally, the 1-smoothness of $F$ follows from the 1-smoothness of $F_i$ by a well-known result on the Lipschitz-continuity of piecewise-continuous functions.
    \item We have
    \begin{align*}
        & F(2, 0) = F_4(2, 0) = \tfrac{16991}{23040},
        && F'(2, 0) = F_4'(2, 0) = (\tfrac{253}{240}, \tfrac{77}{120})^T, \\
        & F(0, 0) = F_1(0, 0) = 0, 
        && F'(0, 0) = F_1'(0, 0) = (0, 0)^T,
    \end{align*}
    and therefore
    \[
        \tfrac{1}{2}\|F'(2,0)-F'(0,0)\|^2=\tfrac{17545}{23040} > \tfrac{16991}{23040} = F(2,0) - F(0,0)- \langle F'(0,0), (2,0)^T \rangle.
    \]
\end{enumerate}
\end{proof}

As an immediate result of the previous claim, there is no 1-smooth convex function $G: \real^2\rightarrow\real$ such that $G(0,0)=F(0,0)$, $G'(0,0)=F'(0,0)$, $G(2,0)=F(2,0)$ and $G'(2,0)=F'(2,0)$ (as such an example will form a contradiction to Theorem~\ref{T:basicinequnconstrained}).
This establishes that the class of smooth and convex functions over an open set is different than the class of unconstrained smooth and convex functions restricted to this set, answering negatively a question raised in~\cite{de2017worst} of whether all smooth and convex functions defined over an open set can be extended to the entire domain.

\section{Properties of convex functions over open sets}
In this section we consider the properties of convex functions defined over an open set. We first show that inequality~\eqref{eq:basicinequnconstrained} holds for any two points that are close enough. Next, we derive a new analytical bound that holds for any two points in the function's domain. Finally, we discuss a stronger bound given as a system of inequalities, which can nevertheless be efficiently evaluated numerically.

\subsection{A local property}
The first main result establishes that locally, a smooth convex function defined over an open set behaves similarly to a smooth and convex function defined over the entire linear space.
\begin{theorem}\label{T:local}
    Let $f:C \rightarrow \real$ be an $L$-smooth and convex function over an open convex set $C \subsetneq \real^d$. Then for any $x,y \in C$ such that $\|x-y\|< \mathrm{dist}(y, \real^d \setminus C)$,
    \begin{equation*}
	    \frac{1}{2L} \|f^\prime(x)-f^\prime(y)\|^2 \leq f(y) - f(x) - \langle f^\prime(x), y - x\rangle.
    \end{equation*}
\end{theorem}

\begin{proof}
The proof is a refinement of the proof in~\cite[Theorem~2.1.5]{Book:Nesterov}.
Denote
\[
	\psi(t) := f(t) - \langle f'(x), t - x \rangle,
\]
then clearly $\psi$ is $L$-smooth, convex and satisfies $\psi'(x)=0$. Furthermore, from the Lipschitz-continuity of $\psi$ we have
\[
	\| \psi'(y) \| = \| \psi'(y) - \psi'(x) \| \leq L \|y-x\| < L \mathrm{dist}(y, \real^d \setminus C),
\]
hence $\mathrm{dist}(y, y - \frac{1}{L} \psi'(y))< \mathrm{dist}(y, \real^d \setminus C)$ which means that $y - \frac{1}{L} \psi'(y)\in C$.
By applying the Descent Lemma~\eqref{eq:descent_lemma} on the points $y$ and $y- \frac{1}{L} \psi'(y)$ we have
\begin{align*}
	& \psi(y - \frac{1}{L} \psi'(y)) - \psi(y) - \langle \psi'(y), y - \frac{1}{L} \psi'(y) - y \rangle \leq \frac{L}{2} \| y - (y - \frac{1}{L} \psi'(y)) \|^2
\end{align*}
or
\begin{align*}
	& \psi(y - \frac{1}{L} \psi'(y)) - \psi(y) - \langle \psi'(y), - \frac{1}{L} \psi'(y) \rangle \leq \frac{L}{2} \| \frac{1}{L} \psi'(y) \|^2
\end{align*}
and we get
\begin{align*}
	& \psi(y - \frac{1}{L} \psi'(y))  \leq \psi(y)  - \frac{1}{2L} \| \psi'(y) \|^2.
\end{align*}
Since $x$ is optimal for $\psi$ (recall that $\psi'(x)=0$), we have
\begin{align*}
	& \psi(x) \left( \leq \psi(y - \frac{1}{L} \psi'(y))\right)  \leq \psi(y)  - \frac{1}{2L} \| \psi'(y) \|^2
\end{align*}
i.e.,
\begin{align*}
	& f(x)  \leq f(y) - \langle f'(x), y - x \rangle  - \frac{1}{2L} \| f'(x) - f'(y) \|^2
\end{align*}
which completes the proof.
\end{proof}

An immediate and useful result follows:
\begin{corollary}\label{C:stepByStep}
Let $f:C \rightarrow \real$ be an $L$-smooth and convex function over an open convex set $C \subsetneq \real^d$ and suppose $x, y\in C$ .
Then for any
\begin{equation}\label{E:boundOnN}
    N > \frac{\|y-x\|}{\min(\mathrm{dist}(x, \real^d \setminus C), \mathrm{dist}(y, \real^d \setminus C))},
\end{equation}
the system of inequalities
\begin{align}
    & \frac{1}{2L} \|f^\prime(x_i)-f^\prime(x_{i+1})\|^2 \leq f(x_i)-f(x_{i+1}) - \langle f^\prime(x_{i+1}), x_i - x_{i+1} \rangle, && i=0,\dots,N-1, \label{C:upperbound}\\
    & \frac{1}{2L} \|f^\prime(x_i)-f^\prime(x_{i+1})\|^2 \leq f(x_{i+1})-f(x_i) - \langle f^\prime(x_{i}), x_{i+1} - x_i \rangle, && i=0,\dots,N-1, \label{C:lowerbound}
\end{align}
holds with
\[
    x_i := x + \frac{i}{N}(y - x), \quad i=0,\dots,N.
\]
\end{corollary}
\begin{proof}
We have
\begin{align*}
    \|x_{i+1} - x_i\| & = \frac{\|y-x\|}{N} < \min(\mathrm{dist}(x, \real^d \setminus C), \mathrm{dist}(y, \real^d \setminus C)) \\
    & \leq \min(\mathrm{dist}(x_i, \real^d \setminus C), \mathrm{dist}(x_{i+1}, \real^d \setminus C)),
\end{align*}
where the last inequality follows since $x_i$ are points along the segment $[x,y]$. We conclude that Theorem~\ref{T:local} applies for all pairs of points $x_i$, $x_{i+1}$, hence~\eqref{C:upperbound} and~\eqref{C:lowerbound} follow.
\end{proof}

\subsection{A global property}
As a consequence of Corollary~\ref{C:stepByStep}, an analytical bound connecting every two points in the set can be established.


\begin{theorem}\label{T:global}
Let $f:C \rightarrow \real$ be an $L$-smooth and convex function over an open convex set $C \subsetneq \real^d$. Then for any $x,y\in C$
\begin{equation}\label{ineq:analytical}
 \frac{\langle f'(y) - f'(x), y - x \rangle^2}{2 L \|y-x\|^2} \leq f(y) - f(x) - \langle f'(x), y - x \rangle.
\end{equation}
\end{theorem}

\begin{proof}
For the sake of simplicity we first establish the result for the case where $f$ is 1-smooth and $f'(x)=0$.

Let $N$ and $x_0, \dots, x_N$ be defined according to Corollary~\ref{C:stepByStep}, then it follows that inequalities~\eqref{C:upperbound} and~\eqref{C:lowerbound} hold.
We start by adding all inequalities in~\eqref{C:lowerbound} together, obtaining
\begin{align*}
    & f(y) - f(x) = f(x_N) - f(x_0) \\
    & \geq \frac{1}{2} \sum_{i=0}^{N-1} \|f^\prime(x_i)-f^\prime(x_{i+1})\|^2  + \frac{1}{N} \sum_{i=0}^{N-1} \langle f^\prime(x_i), y - x \rangle \\
    & = \sum_{i=0}^{N-1}\left( \frac{1}{2} \|f^\prime(x_i)-f^\prime(x_{i+1})\|^2  + \frac{i+1}{N} \langle f^\prime(x_i) - f^\prime(x_{i+1}), y - x \rangle \right) + \langle f^\prime(x_N), y - x \rangle.
\end{align*}

Now, adding~\eqref{C:upperbound} and~\eqref{C:lowerbound} together, we reach
\begin{align}
    & \|f^\prime(x_i)-f^\prime(x_{i+1})\|^2 \leq - \frac{1}{N} \langle f^\prime(x_{i}) - f^\prime(x_{i+1}), y - x \rangle, && i=0,\dots,N-1.\label{E:proof:combined}
\end{align}
Let us denote
\begin{equation}\label{E:proof:alpha}
\begin{aligned}
    \xi &:= N  - \frac{\langle f'(y), y-x \rangle }{\|y-x\|^2}N, \\
    \alpha_i &:= \max\left(0, \xi - i - 1, i - \xi \right), \quad i=0,\dots,N-1,
\end{aligned}
\end{equation}
then multiplying the inequalities in~\eqref{E:proof:combined} by $\alpha_i$ respectively, we get
\begin{align*}
    0 
    & \geq \sum_{i=0}^{N-1} \left( \alpha_i \|f^\prime(x_i)-f^\prime(x_{i+1})\|^2 + \frac{\alpha_i}{N} \langle f^\prime(x_{i}) - f^\prime(x_{i+1}), y - x \rangle\right).
\end{align*}
Finally, adding the two last inequalities together and recalling that $f'(x_0)=f(x)=0$ we have
\begin{align*}
    & f(y) - f(x) \\
    & \geq \sum_{i=0}^{N-1} \left( \frac{2\alpha_i + 1}{2} \|f^\prime(x_i)-f^\prime(x_{i+1})\|^2 + \frac{\alpha_i + i + 1}{N} \langle f^\prime(x_{i}) - f^\prime(x_{i+1}), y - x \rangle\right) + \langle f^\prime(x_N), y - x \rangle \\
    & = \sum_{i=0}^{N-1} \left( \frac{2\alpha_i + 1}{2} \|f^\prime(x_i)-f^\prime(x_{i+1})\|^2 +  \frac{\alpha_i  - \xi + i + 1}{N} \langle f^\prime(x_{i}) - f^\prime(x_{i+1}), y - x \rangle\right) \\&\quad + \frac{N - \xi}{N}\langle f^\prime(x_N), y - x \rangle \\
    & =  \sum_{i=0}^{N-1} \left( \frac{2\alpha_i + 1}{2}\|f^\prime(x_i)-f^\prime(x_{i+1}) + \frac{\alpha_i - \xi + i + 1}{N(2\alpha_i + 1)} (y-x)\|^2 - \frac{(\alpha_i - \xi + i + 1)^2}{2N^2(2\alpha_i + 1)} \|y-x\|^2 \right) \\&\quad + \frac{N - \xi}{N}\langle f^\prime(x_N), y - x \rangle \\
    & \geq - \frac{1}{2N^2} \sum_{i=0}^{N-1} \frac{(\alpha_i - \xi + i + 1)^2}{2\alpha_i + 1} \|y-x\|^2 + \frac{N - \xi}{N}\langle f^\prime(x_N), y - x \rangle \\
    & = -\frac{\langle f'(x_N), y-x \rangle^2}{2\|y-x\|^2} + \frac{\langle f'(x_N), y-x \rangle^2}{\|y-x\|^2} = \frac{\langle f'(x_N), y-x \rangle^2}{2\|y-x\|^2},
\end{align*}
where the sum is evaluated in Lemma~\ref{L:sum} in the appendix.

To complete the proof, consider the general case where the assumptions $L=1$ and $f'(x)=0$ does not necessarily hold. By setting $\phi(z) := \frac{1}{L} (f(z)-\langle f'(x), z - x\rangle)$, we get that $\phi$ satisfies the requirements above and $\phi'(z) = \frac{1}{L}(f'(z)-f'(x))$, hence
\[
    \frac{\langle \phi'(y), y - x\rangle^2}{2\|y-x\|^2} \leq \phi(y) - \phi(x),
\]
which after substitution of the definition of $\phi$ establishes the claim.
\end{proof}

As noted above, the bound~\eqref{ineq:analytical} is strictly stronger than the Descent Lemma~\eqref{eq:descent_lemma}. Indeed, the lower bound is trivially stronger, and regarding the upper bound, applying Theorem~\ref{T:global} with $x$ and~$y$ switched then subtracting $\langle f'(x), y-x\rangle$ from both sides, we reach
\begin{align*}
    f(y) - f(x) - \langle f'(x), y-x\rangle \leq \langle f'(y) - f'(x), y-x\rangle - \frac{\langle f'(y) - f'(x), y - x \rangle^2}{2 L \|y-x\|^2} \\
    \leq \langle f'(y) - f'(x), y-x\rangle \leq \| f'(y) - f'(x)\|\|y-x\| \leq L \|y-x\|^2.
\end{align*}
A comparison between the bounds 
is depicted in Figure~\ref{fig:compare}.

\begin{figure}
    \centering
    \includegraphics[width=0.8\textwidth]{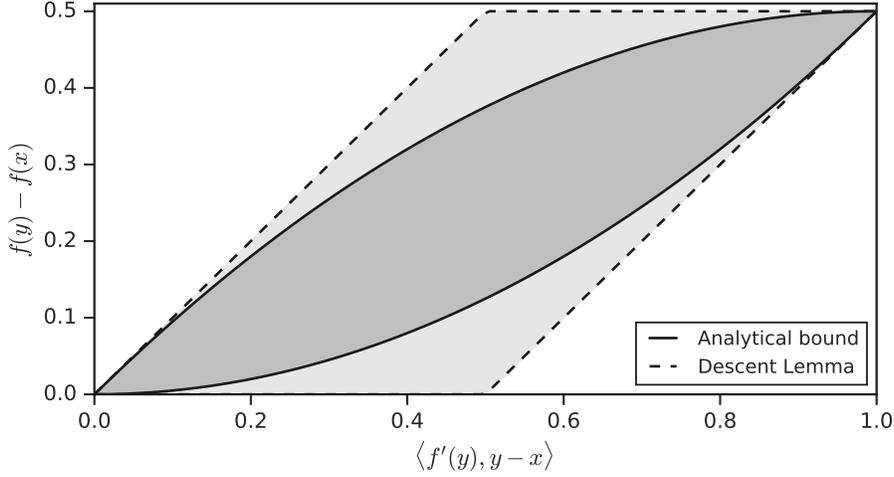}
    \caption{The allowed region for $f(y)-f(x)$ according to the Descent Lemma (outer) and according to Theorem~\ref{T:global} (inner) as a function of $\langle f'(y), y-x\rangle$, when $L=1$, $f'(x)=0$ and $\|y-x\|^2=1$.}
    \label{fig:compare}
\end{figure}

In the next section we proceed to show that even stronger bounds can be derived from Corollary~\ref{C:stepByStep}, and although we are not aware of an analytical form for these bounds, they can nevertheless be efficiently approximated using standard methods.


\subsection{Numerically derived properties}\label{S:numerical}

Here we take an alternative look at Corollary~\ref{C:stepByStep}, observing that this system of inequalities can be viewed as a set of constraints on the function values and gradients.
These constraints can then be used to form bounds by holding some elements in the inequalities as known and optimizing over the unknown elements.



For example, in order to derive bounds on the value of $f(y)$ given that the values of $x$, $f(x)$, $f'(x)$, $y$ and $f'(y)$ are known we 
can take inequalities~\eqref{C:upperbound} and~\eqref{C:lowerbound} as constraints on the value of $f(y)$ treating this value and the values of $f(x_i)$ and $f'(x_i)$ as unknowns, which we denote by $f_i$ and $g_i$, respectively. As established by Corollary~\ref{C:numericalBound} below, an upper bound on the value of $f(y)$ can be obtained by solving the following quadratically-constrained convex optimization problem:
\begin{equation*}
\begin{aligned}
U_N(x, f_x, g_x, y, g_y) :=
        \max_{g_i \in \real^d, f_i\in \real} &\ f_N \\
        \text{s.t. }
        & \textstyle \frac{1}{2L} \|g_i - g_{i+1}\|^2 \leq f_i-f_{i+1} - \frac{1}{N} \langle g_{i+1}, x - y \rangle, && 0\leq i < N , \\
        & \textstyle \frac{1}{2L} \|g_i - g_{i+1}\|^2 \leq f_{i+1}-f_i - \frac{1}{N} \langle g_{i}, y - x \rangle, && 0 \leq i < N, \\
        & f_0 = f_x, \ g_0 = g_x, g_N = g_y.
\end{aligned}
\end{equation*}
Similarly, a lower bound on $f(y)$ can be established by solving $B_N(x, f_x, g_x, y, g_y)$, which is the symmetric problem with a $\min$ operator instead of the $\max$ operator.
\begin{corollary}\label{C:numericalBound}
\begin{enumerate}
    \item 
Let $f:C \rightarrow \real$ be an $L$-smooth and convex function over an open convex set $C \subsetneq \real^d$, and let $x, y\in C$. Then 
    \[
        B_N(x, f(x), f'(x), y, f'(y)) \leq f(y) \leq U_N(x, f(x), f'(x), y, f'(y)),
    \]
    for all values of $N$ satisfying
\[
    N > \frac{\|y-x\|}{\min(\mathrm{dist}(x, \real^d \setminus C), \mathrm{dist}(y, \real^d \setminus C))}.
\]
    \item Suppose $x,y\in \real^d$, $f_x, f_y\in \real$, $g_x, g_y\in \real^d$, $L>0$ and $N>0$ are given such that
\[
    B_N(x, f_x, g_x, y, g_y) \leq f_y \leq U_N(x, f_x, g_x, y, g_y).
\]
    Then there exists a function $F:\real^d\rightarrow \real$ that is $L$-smooth and convex over the segment $[x,y]$, and satisfies
    \begin{align*}
        & F(x) = f_x, \ F(y) = f_y, \\
        & F'(x) = g_x,\ F'(y) = g_y.
    \end{align*}
\end{enumerate}
\end{corollary}
\begin{proof}
\begin{enumerate}
    \item Let $x_i$ be defined as in Corollary~\ref{C:stepByStep}, i.e.,
\[
    x_i := x + \frac{i}{N}(y - x), \quad i=0,\dots,N.
\]
Taking 
\begin{align*}
    & \hat f_i := f(x_i), \quad i=0,\dots,N,\\
    & \hat g_i := f'(x_i), \quad i=0,\dots,N,
\end{align*}
we get from Corollary~\ref{C:stepByStep} that $\{\hat f_i\},\{\hat g_i\}$ are feasible for $U_N(x, f_x, g_x, y, g_y)$. Since the value of the objective at any feasible point is not larger than the optimal solution, we get $f(y) = f(x_N) = \hat f_N \leq U_N(x, f_x, g_x, y, g_y)$, establishing the upper bound. An identical argument establishes the lower bound.
    \item See Appendix~\ref{S:realizableProof}.
\end{enumerate}
\end{proof}


Note that the bound introduced in Corollary~\ref{C:numericalBound} holds for functions that are convex over open sets, however, we have only established that it can be realized by functions that are convex over a line segment.
We conjecture that, under the additional assumption that $U_N(x, f_x, g_x, y, g_y)$ is strictly feasible, the construction presented in Appendix~\ref{S:realizableProof} can be extended to an open set containing the segment, making the bound in Corollary~\ref{C:numericalBound} realizable by a smooth convex function over an open set.

Also note that the corollary above demonstrates the approach for the case where $f(y)$ is unknown while all other properties of $f$ at $x,y$ are known, however, by choosing alternative objective and constraints, the same idea can be generalized to allow finding bounds on any combination of the values $x, f(x), f'(x), y, f'(y)$ with any subset of them being known.

\paragraph{A numerical example}
Convex problems of the form $U_N$ and $B_N$ can be efficiently approximated by interior-point methods~\cite{Book:Nesterov,nesterov1994interior}, allowing the bound in Corollary~\ref{C:numericalBound} to be numerically evaluated given that all relevant quantities are known. 

To illustrate the performance of the bound, consider the case where it is known that $x=0$, $f'(x)=0$, $f(x)=0$, $\|y\|^2 = 1$, $\|f'(y)\|^2=0.5$ and $L=1$; Figure~\ref{fig:compareNumerical} summarizes the allowed range for $f(y)$ obtained from Theorem~\ref{T:global} versus the bounds derived by Corollary~\ref{C:numericalBound} with various values of $N$ and for values of $\langle f'(y), y\rangle \in [0.5, \sqrt{0.5}]$ (this corresponds to the interval where problems $U_N$ and $B_N$ are feasible).
As can be seen in the figure, the numerical bound provides a substantial improvement over the analytical one, especially for lower values of $\langle f'(y), y\rangle$. We also observe that the value of the bound appears to grow very slowly beyond the first few values of $N$, suggesting that a value of $N$ in the range 5--10 is sufficient for obtaining a highly accurate bound. Finally, note that the range for $N=1$ correspond to the unconstrained bound~\eqref{eq:basicinequnconstrained}.




\begin{figure}
    \centering
    \includegraphics[width=0.8\textwidth]{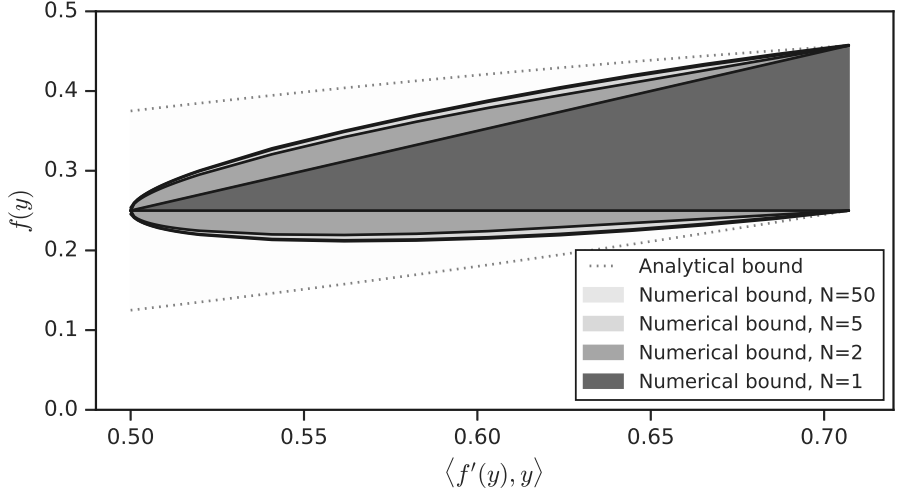}
    \caption{The allowed regions for $f(y)$  according to the analytical bound derived in Theorem~\ref{T:global} (dotted lines) and the numerically-derived bounds, plotted as a function of $\langle f'(y), y\rangle$, for the case $L=1$, $x=0$, $f'(x)=0$, $\|y\|^2=1$ and $\|f'(y)\|^2=0.5$. Note that the regions for $N=5$ and $N=50$ almost completely overlap.}
    \label{fig:compareNumerical}
\end{figure}







\section{Conclusion}

We have established some new properties for the class of smooth convex functions over an open set, showing that this class is essentially different than the class of unconstrained smooth convex functions.

These results emphasize the importance of treating this class independently from the class of unconstrained convex functions as is often done in standard texts.
In particular, regarding optimization methods, these results suggest that there is benefit in designing specialized methods for each of these two classes of functions, as methods designed for the unconstrained case can make additional assumptions not available in the general case.
Moreover, these results show that some care is needed when designing methods for constrained problems, for example, a standard approach for tackling constrained optimization problems is assuming that the problem can be defined by a composite model of the form $F(x)=f(x)+g(x)$, where $f$ is smooth and $g$ is an indicator function that encodes the constraints: this approach, although being natural, limits the applicability of the method to the unconstrained class of functions.

Finally, another interesting question that arises is finding a tight lower complexity bound for the class of constrained functions (in the sense defined by Nemirovsky and Yudin~\cite{nemi-yudi-book83}), and whether it differs significantly from the bound in the unconstrained case that was recently established in~\cite{drori2017exact}.





\appendix

\section{A technical lemma}
\begin{lemma}\label{L:sum}
Suppose $f:C \rightarrow \real$ is a 1-smooth convex function. Further suppose that $x,y\in C$ and that $f'(x)=0$.
If $\xi$, $\alpha_i$ are defined according to~\eqref{E:proof:alpha}, then
\[
    \sum_{i=0}^{N-1} \frac{(\alpha_i- \xi + i + 1)^2}{2\alpha_i + 1} = \frac{\langle f'(y), y-x \rangle^2 }{\|y-x\|^4} N^2.
\]
\end{lemma}
\begin{proof}
First, we observe that
\[
    \alpha_i =
    \begin{cases}
        \xi - i - 1, & 0\leq i < \xi - 1, \\
        0,  & i \in [\xi-1,\xi],\\
        i - \xi, & \xi < i \leq N-1.
    \end{cases}
\]
Now, from the convexity and Lipschitz continuity properties of $f$ we have
\[
    0 \leq \langle f'(y), y-x \rangle = \langle f'(y)-f'(x), y-x \rangle \leq \|y-x\|^2,
\]
which implies $0 \leq \xi \leq N$, and thus the set $\{i: \alpha_i=0,\ 0\leq i\leq N-1\}$ is not empty.
We conclude that there exits some $0\leq N_1\leq N$ such that 
\[
    \alpha_i =
    \begin{cases}
        \xi - i - 1, & 0\leq i < N_1, \\
        0,  & i = N_1, \\
        i - \xi, & N_1 < i \leq N-1.
    \end{cases}
\]
We have
\begin{align*}
    & \sum_{i=0}^{N-1} \frac{(\alpha_i - \xi + i + 1)^2}{2\alpha_i + 1} \\
    & = \sum_{i=0}^{N_1-1} \frac{(\alpha_i - \xi + i + 1)^2}{2\alpha_i + 1} + (N_1 - \xi + 1)^2 +\sum_{i=N_1+1}^{N-1} \frac{(\alpha_i - \xi + i + 1)^2}{2\alpha_i + 1} \\
    & = (N_1 - \xi + 1)^2 +\sum_{i=N_1+1}^{N-1} \frac{(2i - 2 \xi + 1)^2}{2 i - 2\xi + 1} \\
    & = (N_1 - \xi + 1)^2 +2 \sum_{i=N_1+1}^{N-1}  i - (N-N_1-1) \left(2 \xi - 1\right) \\
    & = (N_1 - \xi + 1)^2 + (N-N_1-1) (N+N_1) - (N-N_1-1)\left(2 \xi - 1\right) \\
    & = (\xi - N)^2 = \frac{\langle f'(y), y-x \rangle^2 }{\|y-x\|^4} N^2.
\end{align*}
\end{proof}

\section{Proof of part 2 in Corollary~\ref{C:numericalBound}}\label{S:realizableProof}
Let $\{\tilde f_i,  \tilde g_i\}$ be an optimal solution for $U_N(x, f_x, g_x, y, g_y)$ (this program attains its optimum since its domain is bounded and closed) and set $x_i$ as in the first part of the proof. From the constraints in $U_N$ it follows that for all $0\leq i\leq N-1$
\begin{align*}
    & \frac{1}{2L} \|\tilde g_i-\tilde g_{i+1}\|^2 \leq \tilde f_i-\tilde f_{i+1}-\langle \tilde g_{i+1}, x_i-x_{i+1}\rangle, \\
    & \frac{1}{2L} \|\tilde g_i-\tilde g_{i+1}\|^2 \leq \tilde f_{i+1}-\tilde f_i-\langle \tilde g_i, x_{i+1}-x_i\rangle,
\end{align*}
hence by the convex extension/interpolation theorems of \cite{azagra2017extension,taylor2017smooth} it follows that there exist $L$-smooth and convex functions $F_i:\real^d \rightarrow \real$, $i=0,\dots,N-1$, each extending/interpolating the solution between two adjacent points:
\begin{align*}
    & F_i(x_i) = \tilde f_i,\ F_i(x_{i+1}) = \tilde f_{i+1}, \quad i=0,\dots,N-1, \\
    & F'_i(x_i) = \tilde g_i,\ F_i'(x_{i+1}) = \tilde g_{i+1}, \quad i=0,\dots,N-1.
\end{align*}
Taking $F_u$ to be the piecewise function
\[
    F_u(x) := \{ F_i(x),\ x\in [x_i, x_{i+1}], \quad i=0,\dots,N-1,
\]
it follows from this construction that $F_u$ is continuously differentiable and piecewise $L$-smooth and convex hence, as in the proof of Proposition~\ref{P:counter}, $F_u$ is $L$-smooth and convex. Furthermore, from the constraints in $U_N$ we have 
\begin{align*}
    & F_u(x_0)=F_0(x_0)=\tilde f_0 = f_x,\ F_u'(x_0) = F_0'(x_0) = \tilde g_0 = g_x, \\
    & F_u'(x_N)=F_{N-1}'(x_N)=g_{N}=g_y,
\end{align*}
and from the optimality of $\{\tilde f_i, \tilde g_i\}$, we have $F_u(x_N) = \tilde f_N = U_N(x, f_x, g_x, y, g_y)$. Finally, by Whitney's extension theorem~\cite{whitney1934analytic}, $F_u$ can be extended to the entire space $\real^d$ while keeping its differentiable structure.
Similarly, a function $F_b:\real^d\rightarrow\real$ that is $L$-smooth and convex over $[x,y]$ can be found such that
$F_b(x_0)=\tilde f_0 = f_x$, $F_b'(x_0) = \tilde g_0 = g_x$, $F_b'(x_N)=g_y$ and $F_b(y)=B_N(x, f_x, g_x, y, g_y)$.
Finally, taking a convex linear combination of the two functions $F_u$ and $F_b$, we can reach a function $F$ with the claimed properties.




\bibliographystyle{abbrv}
\bibliography{bib}{}   

\end{document}